\documentclass[10pt]{amsart}
\usepackage{latexsym}
\usepackage{amscd,amssymb,amsmath,amsxtra}
\usepackage[margin=3cm]{geometry}
\usepackage{graphicx}
\usepackage[curve]{xypic}
\usepackage{stmaryrd}

\makeatletter
\let\@wraptoccontribs\wraptoccontribs
\makeatother
\usepackage{amsmath}
\usepackage{epsfig}
\usepackage{graphicx}
\usepackage{eufrak}
\usepackage{dsfont}
\usepackage[english]{babel}
\usepackage{color}
\usepackage{mathabx}

\usepackage{palatino, mathpazo}
\usepackage{amscd}      
\usepackage{amssymb}
\usepackage{dsfont}
\usepackage{xypic}      
\LaTeXdiagrams          
\usepackage[all,v2]{xy}
\xyoption{2cell} \UseAllTwocells \xyoption{frame} \CompileMatrices
\usepackage{latexsym}
\usepackage{epsfig}
\usepackage{enumerate}

\usepackage{latexsym}
\usepackage{epsfig}
\usepackage{amsfonts}
\usepackage{enumerate}
\usepackage{mathrsfs}

\allowdisplaybreaks[3]

\newtheorem{prop}{Proposition}[section]
\newtheorem{lem}[prop]{Lemma}

\newtheorem{cor}[prop]{Corollary}
\newtheorem{thm}[prop]{Theorem}

\newtheorem{defn}[prop]{Definition}

\newenvironment{pf}{\begin{trivlist}\item[]{\sc Proof.}}%
            {\nolinebreak $\Box$ \end{trivlist}}

\newcommand{\noprint}[1]{}

\newcommand{\XX}{{\mathfrak X}}

\renewcommand{\SS}{{\mathfrak S}}

\newcommand{\zz}{{\mathbb Z}}

\newcommand{\qq}{{\mathbb Q}}
\newcommand{\pp}{{\mathbb P}}
\newcommand{\cc}{{\mathbb C}}

\newcommand{\sT}{{\mathcal T}}
\newcommand{\sD}{{\mathcal D}}

\newcommand{\sE}{{\mathcal E}}

\newcommand{\sL}{{\mathcal L}}

\newcommand{\sS}{{\mathcal S}}

\newcommand{\sU}{{\mathcal U}}
\newcommand{\sO}{{\mathcal O}}

\newcommand{\sX}{{\mathcal X}}

\newcommand{\sY}{{\mathcal Y}}
\newcommand{\sM}{{\mathcal M}}
\newcommand{\hH}{{\mathcal H}}

\newcommand{\sZ}{{\mathcal Z}}

\newcommand{\sF}{{\mathcal F}}
\newcommand{\sA}{{\mathcal A}}

\newcommand{\Coh}{\mbox{Coh}}

\newcommand{\sB}{\mathscr{B}}

\DeclareMathOperator{\Hilb}{Hilb}

\DeclareMathOperator{\Pic}{Pic}

\DeclareMathOperator{\Ch}{Ch}
\DeclareMathOperator{\CR}{CR}

\DeclareMathOperator{\SU}{SU}

\DeclareMathOperator{\td}{td}
\DeclareMathOperator{\Stab}{Stab}

\DeclareMathOperator{\age}{age}
\DeclareMathOperator{\orb}{orb}
\DeclareMathOperator{\sst}{st}

\DeclareMathOperator{\Rea}{Re}

\DeclareMathOperator{\Ima}{Im}

\newcommand{\rk}{\mathop{\rm rk}}

\newcommand{\cl}{\mathop{\rm cl}\nolimits}

\newcommand{\tr}{\mathop{\rm tr}\nolimits}

\renewcommand{\Im}{\mathop{\rm Im}}

\newcommand{\rank}{\mathop{\rm rank}\nolimits}

\numberwithin{equation}{subsection}

\newcommand {\mat}      [1] {\left(\begin{array}{#1}}
\newcommand {\rix}          {\end{array}\right)}

\begin{document}
\title[Counting equivariant sheaves on K3 surfaces]{Counting equivariant  sheaves on K3 surfaces}

\author{Yunfeng Jiang}
\address{Department of Mathematics\\ University of Kansas\\ 405 Snow Hall 1460 Jayhawk Blvd\\Lawrence KS 66045 USA}
\email{y.jiang@ku.edu}


\author{Hao Max Sun}
\address{Department of Mathematics, Shanghai Normal University, Shanghai 200234, People's Republic of China}
\email{hsun@shnu.edu.cn, hsunmath@gmail.com}

\sloppy
\begin{abstract}
We study the equivariant sheaf counting theory on K3 surfaces with finite group actions.  Let $\sS=[S/G]$ be a global quotient stack, where $S$ is a K3 surface and $G$ is a finite group acting as symplectic homomorphisms  on $S$.    We show that the Joyce invariants counting Gieseker  semistable sheaves on $\sS$  are independent on the Bridgeland stability conditions.   As an application we prove the  multiple cover formula of Y. Toda  for the  counting invariants for semistable sheaves on local K3 surfaces  with a symplectic finite group action.
\end{abstract}

\maketitle

\tableofcontents

\section{Introduction}

Let $S$ be a smooth projective surface over $\cc$ and $X=S\times \cc$ the local K3 surface.  The sheaf counting theory of $S$ and $X$, such as Donaldson-Thomas invariants and Pandharipande-Thomas invariants, has
rich geometric structures.   In \cite{Toda_Adv},  \cite{Toda_JDG}, Y. Yoda studied the Joyce  counting invariants of semistable objects in the derived category defined by Bridgeland stability conditions.  Especially in \cite{Toda_Adv}, Toda proved that the Joyce invariants are independent to the stability conditions, and he used this result to study the Pandharipande-Thomas  stable pair invariants of local K3 surfaces in \cite{Toda_JDG}.
A multiple cover formula for the Joyce invariants for $X$ was conjectured in \cite{Toda_JDG}, and proved in \cite{MT} using both the Gromov-Witten theory of local K3 surfaces and the  Pandharipande-Thomas  stable pair theory.
The multiple cover formula is essential to the calculation of the generating series of the $\SU(r)$-Vafa-Witten invariants for K3 surfaces in \cite{TT2}.

The multiple cover formula of Toda was generalized and proved to be true for  \'etale K3 gerbes in \cite{JT}, which is essential to calculate the  $\SU(r)/\zz_r$-Vafa-Witten invariants for K3 surfaces in \cite{JT2}.
Both the generating series of  the $\SU(r)$-Vafa-Witten invariants and  $\SU(r)/\zz_r$-Vafa-Witten invariants for K3 surfaces are the two sides inspired by the S-duality conjecture in \cite{VW}.
The S-duality conjecture of Vafa-Witten  for K3 surfaces  was proved  in prime rank  by \cite{Jiang_2019}, \cite{JK}, and all ranks in \cite{JT2}.

In this paper we study the sheaf counting theory for local
orbifold K3 surfaces, and prove a multiple cover formula for the Joyce invariants counting semistable sheaves on local orbifold K3 surfaces.   The formula is essential to calculating the generating function of the $\SU(r)/\zz_r$-Vafa-Witten invariants for
orbifold K3 surfaces in \cite{JKT}.

\subsection{Sheaf counting on $[S/G]$}
Let $S$ be a smooth projective K3 surface and $G$ a finite group acting on $S$ as symplectic morphisms.  We consider the surface Deligne-Mumford stack $\sS=[S/G]$.  The action of $G$ on $S$ only has isolated fixed points, and the classification of the number of fixed points corresponding to different finite group actions are given in \cite[Chapter 15]{Huybrechts}.   Thus all the stacky locus of  $\sS=[S/G]$ are stacky points, and they are given by $ADE$ type singularities on the coarse moduli space $\overline{\sS}=S/G$.  Let $\sigma:  Y\to S/G$ be the minimal resolution.  It is a crepant resolution and the exceptional curves over a singular point in $S/G$ are given by $ADE$ type Dynkin diagrams.

Let $\Coh(\sS)$ be the abelian category of coherent sheaves on $\sS$, which is by  definition the category of $G$-equivariant sheaves on $S$.  We fix an ample divisor $\omega$ on $\overline{\sS}$.
Similar to the K3 surface, we define the orbifold Mukai vector $v_{\orb}(E)$ for any object $E\in \Coh(\sS)$ as:
$$v_{\orb}(E):=\widetilde{\Ch}(E)\cdot \sqrt{\widetilde{\td}_{\sS}}
=(\widetilde{\Ch}_0(E), \widetilde{\Ch}_1(E), \widetilde{\Ch}_0(E)+\widetilde{\Ch}_2(E))\in H^*_{\CR}(\sS)$$
where $\widetilde{\Ch}$ is the orbifold Chern character and $\widetilde{\td}$ is the orbifold Todd class.
Here $H^*_{\CR}(\sS)$ is the Chen-Ruan cohomology of $\sS$.  Let
$\Gamma_0^G$ be the group of all the orbifold Mukai vectors.

We take
$\sM(\Coh(\sS))$ to be the moduli stack of coherent sheaves on $\sS$, which is an algebraic stack locally of finite type over $\cc$.
We work on the moduli stack $\sM_{\omega}(v_{\orb})$ of $\omega$-Gieseker semistable sheaves
$E\in\Coh(\sS)$ with $v_{\orb}(E)=v_{\orb}\in \Gamma_0^G$.
In general, one takes a generating sheaf $\Xi$ on $\sS=[S/G]$ as in \cite{Nironi} and define the moduli stack of semistable sheaves on $\sS=[S/G]$
using the modified Gieseker stability defined by $\Xi$.  Since $\sS=[S/G]$ only has isolated stacky points, the  modified Gieseker stability  is equivalent to the $\omega$-Gieseker stability.

There is a Hall algebra structure on $\Coh(\sS)$ and we denote it by $\hH(\Coh(\sS))$, see \S \ref{subsec_Hall_algebra}.  Thus we have an element
$$\delta_{\omega,\sS}:=[\sM_{\omega}(v_{\orb})\hookrightarrow \sM(\Coh(\sS))]\in \hH(\Coh(\sS))$$
We take its logarithm as:
\begin{equation}\label{eqn_epsilon_sS_intro}
\epsilon_{\omega, \sS}(v_{\orb}):=
\sum_{\substack{\ell\ge 1, v_1+\cdots+v_{\ell}=v_{\orb}, v_i\in \Gamma_0^G \\
\overline{\chi}_{\omega, v_i}(m)=\overline{\chi}_{\omega,v_{\orb}}(m)}}
\frac{(-1)^{\ell-1}}{\ell}\delta_{\omega, \sS}(v_1)\star\cdots \star \delta_{\omega, \sS}(v_{\ell})
\end{equation}
where $\overline{\chi}_{\omega, v_{\orb}}(m)$ is the reduced Hilbert polynomial.

Let
$$C(\sS):=\Im(v_{\orb}: \Coh(\sS)\to \Gamma_0^G).$$
Then the Joyce invariants are defined as follows.
If  $v_{\orb}\in C(\sS)$, we define
\begin{equation}\label{Joyce_invariants_sS_intro}
J^{\omega}(v_{\orb})=\lim_{q^{\frac{1}{2}}\to 1}(q-1)P_q(\epsilon_{\omega,\sS}(v_{\orb}))
\end{equation}
where $P_q(-)$ is the Poincar\'e polynomial of the stack.  The Joyce invariants are independent to the polarization $\omega$.

Generalizing the construction in \cite{Bridgeland_K3}, there exist Bridgeland stability conditions $\sigma=(\sZ, \sA)$ on the bounded derived category $D^b(\Coh(\sS))$  of coherent sheaves on $\sS$.
Here Bridgeland stability condition $\sigma=(\sZ, \sA)$ is a pair, where $\sA$ is the heart of a bounded $t$-structure on $D^b(\Coh(\sS))$, and $\sZ: K(\sA)\to \cc$ is the central charge satisfying certain conditions, see \S \ref{subsec_Bridgeland_stability} for more details.
We define $\sM^{\sigma}(v_{\orb})$ to be the moduli stack of $\sigma$-semistable objects  $E\in \sA$ with $v_{\orb}(E)=v_{\orb}$.
The heart $\sA$ is an abelian category.
Let $\hH(\sA)$ be the Hall algebra of $\sA$.  Then similar to (\ref{eqn_epsilon_sS_intro}), we replace the moduli stack there by the moduli stack of $\sigma$-semistable objects in $\sA$ and define the Joyce invariants
$J^{\sigma}(v_{\orb})$ in the same way.
Our first main result is:

\begin{thm}\label{thm_Joyce_sigma_intro}(Theorem \ref{thm_Joyce_sigma})
The Joyce invariant $J^{\sigma}(v_{\orb})$ is independent to the stability conditions. Moreover
$$J^{\sigma}(v_{\orb})=J^{\omega}(v_{\orb}).$$
\end{thm}
We prove Theorem \ref{thm_Joyce_sigma_intro} in the remaining subsections in \S \ref{sec_quotient_K3}.

\subsection{Multiple cover formula}

We generalize the above counting invariants to the local orbifold K3 surfaces. Let $\XX:=\sS\times \cc$ be the local orbifold K3 surface, which is a Calabi-Yau threefold Deligne-Mumford stack. Let
$\pi: \XX\to \cc$ be the projection and let
$\Coh_{\pi}(\XX)\subset \Coh(\XX)$ be the subcategory of coherent sheaves supported on the fibers of $\pi$.
Let $\sM(\Coh_{\pi}(\XX))$  be the moduli stack of objects in $\Coh_{\pi}(\XX)$, which is an algebraic stack locally of finite type.
Still fix the polarization $\omega$, and let $\sM_{\omega,\XX}(v_{\orb})$ be the moduli stack of $\omega$-Gieseker semistable sheaves
$E\in\Coh_{\pi}(\XX)$ with $v_{\orb}(E)=v_{\orb}\in \Gamma_0^G$.
Let  $\hH(\Coh_{\pi}(\XX))$ be the Hall algebra.  We have
$$\delta_{\omega,\XX}:=[\sM_{\omega,\XX}(v_{\orb})\hookrightarrow \sM(\Coh_{\pi}(\XX))]\in \hH(\Coh_{\pi}(\XX))$$
Its logarithm is:
\begin{equation}\label{eqn_epsilon_XX_intro}
\epsilon_{\omega, \XX}(v_{\orb}):=
\sum_{\substack{\ell\ge 1, v_1+\cdots+v_{\ell}=v_{\orb}, v_i\in \Gamma_0^G \\
\overline{\chi}_{\omega, v_i}(m)=\overline{\chi}_{\omega,v_{\orb}}(m)}}
\frac{(-1)^{\ell-1}}{\ell}\delta_{\omega, \XX}(v_1)\star\cdots \star \delta_{\omega, \XX}(v_{\ell})
\end{equation}
The Joyce invariants counting semistable coherent sheaves in $\Coh_{\pi}(\XX)$ are given by
$$J_{\XX}^{\omega}(v_{\orb})=\lim_{q^{\frac{1}{2}}\to 1}(q-1)P_q(\epsilon_{\omega,\XX}(v_{\orb}))\in\qq.$$
Our second result is the multiple cover formula of such invariants:

\begin{thm}\label{thm_multiple_cover_intro}(Theorem \ref{thm_multiple_cover})
We have  a multiple cover formula for $J_{\XX}(v_{\orb})$:
$$J_{\XX}(v_{\orb})=\sum_{k|v_{\orb}, k\ge 1}
\frac{1}{k^2}\chi(\Hilb^{n, \mathfrak{m}}(\sS))
$$
where  the data $n, \mathfrak{m}$ are determined by
$\frac{1}{k}v_{\orb}=(r, (\beta, \mathfrak{m}), n)$, and $\Hilb^{n, \mathfrak{m}}(\sS)$ is the Hilbert scheme of the orbifold K3 surface
$\sS$ with data $(n, \mathfrak{m})$.
\end{thm}

We prove Theorem \ref{thm_multiple_cover_intro} by working on the sheaf counting invariants on the compactification
$\overline{\XX}=\sS\times\pp^1$ and $\overline{Z}=Y\times\pp^1$.
For the crepant resolution $Y\to S/G$.  From \cite{BKR}
There exists a  derived equivalence
$$\Phi: D(\Coh(\sS))\stackrel{\sim}{\longrightarrow} D(\Coh(Y)).$$
The equivalence induces the following  equivalence:
$$\Phi: D(\Coh_{\pi}(\overline{\XX}))\stackrel{\sim}{\longrightarrow} D(\Coh_{\pi}(\overline{Z}))$$
where $\Coh_{\pi}(\overline{\XX})\subset \Coh(\overline{\XX})$ (rep. $\Coh_{\pi}(\overline{Z})\subset \Coh(\overline{Z})$) is the subcategory of
coherent sheaves on $\overline{\XX}$ (rep. $\overline{Z}$) supported on the fibers under $\pi: \overline{\XX}\to \pp^1$ (rep. $\pi: \overline{Z}\to \pp^1$).
Then we define the Joyce invariants  $\overline{J}_{\overline{\XX}}(v_{\orb})$ (rep.  $\overline{J}_{\overline{Z}}(v_{Y})$) on the categories above.  The invariants satisfy the relation
 $\overline{J}_{\overline{\XX}}(v_{\orb})=2 J_{\XX}(v_{\orb})$.
We prove the automorphic property for the invariants $J_{\XX}(v_{\orb})$, and show that it is the same as $J_{Z}(v_Y)$ under the
isomorphism $\Phi_*: H^*_{\CR}(\sS)\to H^*(Y)$ such that $\Phi_*(v_{\orb})=v_Y$.   The multiple cover formula for the invariants
$J_{Z}(v_Y)$ in \cite{Toda_JDG}, \cite{MT}, plus the properties of the Hilbert scheme $\Hilb(\sS)$ of zero dimensional substacks  in $\sS$ (see \cite{BG}) implies the formula in Theorem \ref{thm_multiple_cover_intro}.

\subsection{Outline}
Here is the short outline of the paper.   We study the sheaf counting invariants on the orbifold K3 surface $[S/G]$ in \S \ref{sec_quotient_K3}, where we prove Theorem \ref{thm_Joyce_sigma_intro}.
The Joyce invariants counting semistable sheaves on the local orbifold K3 surfaces are defined in \S \ref{sec_local_orbifold_K3}, and we prove Theorem
\ref{thm_multiple_cover_intro}.

\subsection{Convention}
We work over complex number  $\cc$   throughout of the paper.    We use Roman letter $E$ to represent a coherent sheaf on a projective DM stack $\sS$, and use curly letter $\sE$ to represent the sheaves on the total space
Tot$(\sL)$ of a line bundle $\sL$ over $\SS$.
We reserve {\em $\rk$} for the rank of the torsion free coherent sheaves $E$.

\subsection*{Acknowledgments}

Y. J. would like to thank Amin Gholampour, Martijn Kool, and Richard
Thomas for valuable discussions on the Vafa-Witten invariants. Y. J.
is partially supported by Simons foundation collaboration grant. H.
M. S. is partially supported by National Natural Science Foundation
of China (Grant No. 11771294, 11301201).


\section{The Joyce invariants for  quotient K3 surfaces}\label{sec_quotient_K3}

\subsection{Symplectic action on  K3 surfaces}\label{subsec_quotient_K3}

Let $S$ be a smooth projective K3 surface, and $G$ a finite group which acts on $S$ as symplectic morphisms.
We consider the surface Deligne-Mumford stack $\sS:=[S/G]$. From \cite[Chapter 15]{Huybrechts}, the action only has isolated fixed points which are all ADE type singularities on the coarse moduli space
$\overline{\sS}=S/G$.  For instance, when $G=\mu_2$, there are totally $8$ isolated $A_1$-type singularities, and this is called the Nikulin involution.
More basic properties of the symplectic action on the K3 surface $S$ can be found in  \cite[Chapter 15]{Huybrechts}.

\subsection{Mukai lattice}

For the K3 surface $S$ with $H^1(S,\sO_S)=0$, the Mukai lattice is define by:
$$\widetilde{H}(S,\zz):=H^0(S,\zz)\oplus H^2(S,\zz)\oplus H^4(S,\zz)$$
and for $v_i=(r_i, \beta_i, n_i)\in \widetilde{H}(S,\zz)$ ($i=1,2$), the Mukai paring is defined as:
\begin{equation}\label{eqn_Mukai_paring}
\langle v_1, v_2\rangle=\beta_1\beta_2-r_1n_2-r_2n_1.
\end{equation}
There is a weight two Hodge structure on $\widetilde{H}(S,\zz)\otimes\cc$, which is given by:
$$
\begin{cases}
\widetilde{H}^{2,0}(S):=H^{2,0}(S);\\
\widetilde{H}^{1,1}(S):=H^{0,0}(S)\oplus H^{1,1}(S)\oplus H^{2,2}(S);\\
\widetilde{H}^{0,2}(S):=H^{0,2}(S).
\end{cases}
$$
Define:
$$\Gamma_0:=\widetilde{H}(S,\zz)\cap \widetilde{H}^{1,1}(S)=\zz\oplus NS(S)\oplus \zz.$$
The Mukai vector $v(E)\in \Gamma_0$ for any $E\in D^b(\Coh(S))$ is defined by
$$v(E):=\Ch(E)\sqrt{\td_S}=(\Ch_0(E), \Ch_1(E), \Ch_0(E)+\Ch_2(E)).$$
Riemann-Roch theorem tells us that
$$\chi(E,F)=-\langle v(E), v(F)\rangle.$$

\subsection{Orbifold Mukai vector}

Let $\sS=[S/G]$ be the surface Deligne-Mumford stack given by the symplectic action of $G$ on $S$.  We consider the Chen-Ruan cohomology
$$H^{*}_{\CR}(\sS,\qq)=\bigoplus_{i\in I_{\sS}}H^{*-2\age(\sS_i)}(\sS_i,\qq)$$
where the inertia stack $I\sS=\bigsqcup_{i\in I_{\sS}}\sS_i$ is the decomposition of the inertia stack of $\sS$ and
$I_{\sS}$ is the finite  index set.  We let $\sS_0:=\sS$ corresponding to the non-twisted sector.  All of the other components $\sS_i$ for $i\neq 0$ are isolated stacky point
$BG_{\sS_i}$ where $G_{\sS_i}\subset \SU(2)$ is a finite subgroup of $\SU(2)$ which corresponds to the centralizer of the conjugacy class $(g)$ in the local isotropy group of the stacky point.
The component $\sS_i$ always have age $1$.
We let $I_1\sS\subset I\sS$ be the components in the inertia stack containing all twisted sectors (which are all stacky ADE type points).

We define a similar weight two Hodge structure on
$$\widetilde{H}_{\CR}(\sS)=H^0_{\CR}(\sS)\oplus H^2_{\CR}(\sS)\oplus H^4_{\CR}(\sS)$$
which is given by:
$$
\begin{cases}
\widetilde{H}^{2,0}_{\CR}(\sS):=H^{2,0}(\sS);\\
\widetilde{H}_{\CR}^{1,1}(\sS):=H^{0,0}(\sS)\oplus H^{1,1}(\sS)\oplus H^0(I_1\sS)\oplus H^{2,2}(\sS);\\
\widetilde{H}_{\CR}^{0,2}(\sS):=H^{0,2}(\sS).
\end{cases}
$$
We define
$$\Gamma_0^{G}:=\widetilde{H}_{\CR}(\sS)\cap \widetilde{H}_{\CR}^{1,1}(\sS)
=\qq\oplus H^{1,1}(\sS)\oplus \qq^{|I_1\sS|}\oplus \qq
=\qq\oplus NS(\sS)\oplus \qq^{|I_1\sS|}\oplus \qq$$
where $|I_1\sS|$ is the number of components in $I_1\sS$.  We define the orbifold Mukai vector $v_{\orb}(E)$ for any
$E\in D^b(\Coh(\sS))$ as
$$v_{\orb}(E):=\widetilde{\Ch}(E)\cdot \sqrt{\widetilde{\td}_{\sS}}
=(\widetilde{\Ch}_0(E), \widetilde{\Ch}_1(E), \widetilde{\Ch}_0(E)+\widetilde{\Ch}_2(E))$$
where $\widetilde{\Ch}: K(\Coh(\sS))\to H^*_{\CR}(\sS)$ is the orbifold Chern character.
The orbifold Riemann-Roch theorem (for instance \cite{CIJS}) says that
$$\chi(E,F)=-\langle v_{\orb}(E), v_{\orb}(F)\rangle$$
for $E, F\in D^b(\Coh(\sS))$.

\subsection{Hall algebra}\label{subsec_Hall_algebra}

We talk about the counting sheaf  invariants  on the derived category $D^b(\Coh(\sS))$ of coherent sheaves on $\sS$.  It is well known that
$D^b(\Coh(\sS))$ is the same as the derived category of $G$-equivariant sheaves on the K3 surface $S$. Let us denote by
$\sM(\Coh(\sS))$ the moduli stack of coherent sheaves on $\sS=[S/G]$.  The stack $\sM(\Coh(\sS))$ is an algebraic stack locally of finite type over $\cc$.
We fix a ample divisor $\omega$ on $\overline{\sS}=S/G$.
Let $v_{\orb}\in \Gamma_0^G$ and
$$\sM_{\omega, \sS}(v_{\orb})\subset \sM(\Coh(\sS))$$
the substack of $\omega$-Gieseker semistable sheaves $E\in \Coh(\sS)$ satisfying $v_{\orb}(E)=v_{\orb}$.

Let $\hH(\Coh(\sS))$ be the $\qq$-vector space spanned by the isomorphism classes of symbols:
$$[\sX\stackrel{f}{\rightarrow}\sM(\Coh(\sS))]$$
where $\sX$ is an algebraic stack of finite type over $\cc$ with affine stabilizers and $f$ is a morphism of stacks.  The relations are given by:
\begin{equation}
[\sX\stackrel{f}{\rightarrow}\sM(\Coh(\sS))]-[\sY\stackrel{f|_{\sY}}{\rightarrow}\sM(\Coh(\sS))]-[\sU\stackrel{f|_{\sU}}{\rightarrow}\sM(\Coh(\sS))]
\end{equation}
where $\sY\subset \sX$ is a closed substack and $\sU=\sX\setminus \sY$.
There is a Hall algebra structure on  $\hH(\Coh(\sS))$ given by the Hall algebra product $\star$
\begin{equation}
[\sX\stackrel{f}{\rightarrow}\sM(\Coh(\sS))]\star [\sY\stackrel{g}{\rightarrow}\sM(\Coh(\sS))]=
[\sZ\stackrel{p_2\circ h}{\rightarrow}\sM(\Coh(\sS))]
\end{equation}
where $\sZ$ and $h$ fit into the Cartisian diagram:
$$
\xymatrix{
\sZ\ar[r]^{h}\ar[d]& \sE xt(\Coh(\sS))\ar[r]^{p_2}\ar[d]^{(p_1,p_2)}& \sM(\Coh(\sS))\\
\sX\times\sY\ar[r]^-{f\times g}& \sM(\Coh(\sS))\times \sM(\Coh(\sS))
}
$$
where $\sE xt(\Coh(\sS))$ is the stack of short exact sequences in $\Coh(\sS)$ and
$$p_i: \sE xt(\Coh(\sS))\to \sM(\Coh(\sS)) (i=1,2,3)$$
are morphisms of stacks sending a short exact sequence
$$0\to E_1\rightarrow E_2\rightarrow E_3\to 0$$
to the objects $E_i$ respectively.

\subsection{Joyce invariants}\label{subsec_Joyce_invariant}

Joyce \cite{Joyce_2007} defined a morphism
$$P_q: \hH(\Coh(\sS))\to \qq(q^{\frac{1}{2}})$$
such that if $H$ is a special algebraic group acting on a scheme $Y$, we have  from \cite[Definition 2.1]{Joyce_2007}
$$P_q\left([Y/H]\stackrel{f}{\rightarrow} \sM(\Coh(\sS))\right)=P_q(Y)/P_q(H)$$
where $P_q(Y)$ is the virtual Poincar\'e polynomial of $Y$.

We define an element
$$\delta_{\omega, \sS}(v_{\orb}):=[\sM_{\omega, \sS}(v_{\orb})\hookrightarrow \sM(\Coh(\sS))]
\in \hH(\Coh(\sS))$$
and its logarithm as:
\begin{equation}\label{eqn_epsilon_sS}
\epsilon_{\omega, \sS}(v_{\orb}):=
\sum_{\substack{\ell\ge 1, v_1+\cdots+v_{\ell}=v_{\orb}, v_i\in \Gamma_0^G \\
\overline{\chi}_{\omega, v_i}(m)=\overline{\chi}_{\omega,v_{\orb}}(m)}}
\frac{(-1)^{\ell-1}}{\ell}\delta_{\omega, \sS}(v_1)\star\cdots \star \delta_{\omega, \sS}(v_{\ell})
\end{equation}
where $\overline{\chi}_{\omega, v_{\orb}}(m)$ is the reduced Hilbert polynomial, i.e.,
$$\overline{\chi}_{\omega, v_{\orb}}(m)=\frac{\chi_{\omega, v_{\orb}}(m)}{a_d} (a_d \text{~is the first coefficient~})$$

\begin{defn}\label{defn_Joyce_invaraints_sS}
Let
$$C(\sS):=\Im(v_{\orb}: \Coh(\sS)\to \Gamma_0^G).$$
Then if  $v_{\orb}\in C(\sS)$, we define $J^{\omega}(v_{\orb})\in \qq$ as:
$$J^{\omega}(v_{\orb})=\lim_{q^{\frac{1}{2}}\to 1}(q-1)P_q(\epsilon_{\omega,\sS}(v_{\orb}));$$

If  $-v_{\orb}\in C(\sS)$, we define $J^{\omega}(v_{\orb})=J^{\omega}(-v_{\orb})$  and
if $\pm v_{\orb}\notin C(\sS)$, then $J^{\omega}(v_{\orb})=0$.
\end{defn}

\cite[Theorem 6.2]{Joyce_2007} guarantees that the limit above exists.

\subsection{Bridgeland stability conditions}\label{subsec_Bridgeland_stability}

We will construct the Bridgeland stability conditions on
$D^b(\Coh(\sS))$.

\begin{defn}
A Bridgeland stability condition on $D^b(\Coh(\sS))$ is a pair
$\sigma=(Z, \sA)$ consisting of the heart of a bounded $t$-structure
$\sA\subset D^b(\Coh(\sS))$ and a group homomorphism map (called
central charge) $Z: K(\sA)\to \cc$ such that the following are
satisfied:

\begin{enumerate}
\item $Z$ satisfies the following positivity property for any $0\neq E\in
\sA$:
$$Z(E)\in\{re^{i\pi\phi}: r>0, 0<\phi\leq1\}.$$

\item Every
object of $\sA$ has a Harder-Narasimhan filtration in $\sA$ with
respect to $\nu_{\sigma}$-stability, here the slope $\nu_{\sigma}$
of an object $E\in \mathcal{A}$ is defined by
\begin{eqnarray*}
\nu_{\sigma}(E)= \left\{
\begin{array}{lcl}
+\infty,  & &\mbox{if}~\Ima Z(E)=0,\\
&&\\
-\frac{\Rea Z(E)}{\Ima Z(E)}, & &\mbox{otherwise}.
\end{array}\right.
\end{eqnarray*}

\item $\sigma$ satisfies the support property: $Z$ factors as
$K(\sA)\xrightarrow{v}\Lambda\xrightarrow{g}\cc$ where $\Lambda$ is
a finitely generated free abelian group, and there exists a
quadratic form $Q$ on $\Lambda_{\mathbb{R}}$ such that
$Q|_{\ker(g)}$ is negative definite and $Q(v(E))\geq0$ for any
$\nu_{\sigma}$-semistable object $E\in\sA$.
\end{enumerate}
\end{defn}

We say $E\in\mathcal{A}$ is $\nu_{\sigma}$-(semi)stable if for any
non-zero subobject $F\subset E$ in $\mathcal{A}$, we have
$$\nu_{\sigma}(F)<(\leq)\nu_{\sigma}(E/F).$$
The Harder-Narasimhan filtration of an object $E\in \mathcal{A}$ is
a chain of subobjects
$$0=E_0\subset E_1\subset\cdots\subset E_m=E$$ in $\mathcal{A}$ such
that $G_i:=E_i/E_{i-1}$ is $\nu_{\sigma}$-semistable and
$\nu_{\sigma}(G_1)>\cdots>\nu_{\sigma}(G_m)$. We set
$\nu_{\sigma}^+(E):=\nu_{\sigma}(G_1)$ and
$\nu_{\sigma}^-(E):=\nu_{\sigma}(G_m)$.

We now give a construction of Bridgeland stability conditions on
$\sS=[S/G]$. For a fixed $\mathbb{Q}$-divisor $D$ on $\mathcal{S}$,
we define the twisted Chern character $\Ch^{D}(E)=e^{-D}\Ch(E)$ for
any $E\in D^b(\mathcal{S})$. More explicitly, we have
\begin{eqnarray*}
\begin{array}{lcl}
\Ch^{D}_0=\Ch_0=\rank  && \Ch^{D}_2=\Ch_2-D\Ch_1+\frac{D^2}{2}\Ch_0\\
&&\\
\Ch^{D}_1=\Ch_1-D\Ch_0 &&
\Ch^{D}_3=\Ch_3-D\Ch_2+\frac{D}{2}\Ch_1-\frac{D^3}{6}\Ch_0.
\end{array}
\end{eqnarray*}

Let $\omega$ be an ample divisor on $\sS$. We define the twisted
slope $\mu_{\omega, D}$ of a coherent sheaf $E\in \Coh(\mathcal{S})$
by
\begin{eqnarray*}
\mu_{\omega, D}(E)= \left\{
\begin{array}{lcl}
+\infty,  & &\mbox{if}~\Ch^D_0(E)=0,\\
&&\\
\frac{\omega\Ch_1^{D}(E)}{\omega^2\Ch_0^{D}(E)}, &
&\mbox{otherwise}.
\end{array}\right.
\end{eqnarray*}
There exists Harder-Narasimhan filtration
$$0=E_0\subset E_1\subset\cdots\subset E_n=E$$
such that $E_i/E_{i-1}$ is $\mu_{\omega, D}$-semistable and
$\mu_{\omega, D}(E_1/E_0)>\cdots> \mu_{\omega, D}(E_n/E_{n-1})$. We
let $\mu_{\omega, D}^+(E):=\mu_{\omega, D}(E_1/E_0)$ and
$\mu_{\omega, D}^-(E):= \mu_{\omega, D}(E_n/E_{n-1})$. We define
\begin{align}\label{eqn_torsion_pair}
&\sF_{\omega, D}:=\{E\in \Coh(\sS)| \mu_{\omega, D}^+(E)\leq 0\} \\
&\sT_{\omega, D}:=\{E\in \Coh(\sS)| \mu_{\omega, D}^-(E)> 0\}
\nonumber
\end{align}
We let $\sA_{\omega, D}\subset D^b(\mathcal{S})$ be the
extension-closure $\langle \sF_{\omega, D}[1], \sT_{\omega,
D}\rangle $. By the general theory of torsion pairs and tilting
\cite{HRS}, $\sA_{\omega, D}$ is the heart of a bounded t-structure
on $D^b(\mathcal{S})$; in particular, it is an abelian category. Let
$k>0$ be a positive rational number, and consider the following
central charge
$$Z_{k,D}(E)=-\Ch^D_2(E)+\frac{k^2}{2}\omega^2 \Ch^D_0(E)+i \omega
\Ch^D_1(E),$$ where $E\in \sA_{\omega, D}$.

\begin{thm}\label{Bri}
For any $(k, D)\in \mathbb{Q}_{>0}\times
NS(\mathcal{S})_{\mathbb{Q}}$, $\sigma_{k,D}=(Z_{k, D}, \sA_{\omega,
D})$ is a Bridgeland stability condition on $\mathcal{S}$.
\end{thm}
\begin{pf}
By the following Hodge index theorem and Bogomolov's inequality on
$\sS$, one sees that
$$\overline{\Delta}^D_{\omega}(E):=(\omega\Ch^D_1(E))^2-2\omega^2\Ch^D_0(E)\Ch^D_2(E)\geq0$$ for any $\mu_{\omega, D}$-semistable sheaf.
Therefore the required assertion is proved in \cite{Bridgeland_K3},
\cite{AB} and \cite[Appendix 2]{BMS}. See also \cite[Corollary
2.22]{PT}.
\end{pf}
\begin{thm}[Hodge index theorem]
Let $L$ be a divisor on $\mathcal{S}$, then we have
$$L^2\omega^2\leq (L\omega)^2.$$
\end{thm}

\begin{thm}[Bogomolov's inequality]\label{Bog}
Let $E$ be a torsion free $\mu_{\omega, D}$-semistable sheaf on
$\mathcal{S}$. Then we have
$$\Delta(E):=(\Ch^D_1(E))^2-2\Ch^D_0(E)\Ch^D_2(E)\geq0.$$
\end{thm}
\begin{proof}
See \cite{JK}.
\end{proof}

The stability condition $\sigma_{k,D}$ above is usually called
tilt-stability or $\nu_{k, D}$-stability, where the $\nu_{k,
D}$-slope of an object $E\in \sA_{\omega,D}$ is defined by
\begin{eqnarray*}
\nu_{k, D}(E)= \left\{
\begin{array}{lcl}
+\infty,  & &\mbox{if}~\omega\Ch_1^D(E)=0,\\
&&\\
\frac{\Ch_2^D(E)-\frac{k^2}{2}\omega^{2}\Ch_0^D(E)}{\omega\Ch_1^D(E)},
& &\mbox{otherwise}.
\end{array}\right.
\end{eqnarray*}
The $\nu_{k, D}$-semistable objects still satisfy Bogomolov's
inequality:

\begin{thm}\label{Bog2}
Let $E$ be a $\nu_{k, D}$-semistable object in $\sA_{\omega,D}$.
Then we have
$$\overline{\Delta}^D_{\omega}(E):=(\omega\Ch^D_1(E))^2-2\omega^2\Ch^D_0(E)\Ch^D_2(E)\geq0.$$
\end{thm}
\begin{proof}
The proof is the same as that of \cite[Theorem 3.5]{BMS}.
\end{proof}

\subsection{The moduli stack}

For the Bridgeland stability condition $\sigma=(\sZ, \sA_{\omega})$, let $\sM^{v_{\orb}}(\sigma)$ be the moduli stack of $\sigma$-semistable objects $E\in D(\Coh(\sS))$ with
$v_{\orb}(E)=v_{\orb}$.  If $G=1$, Toda in \cite{Toda_JDG}, \cite[\S 3]{Toda_Adv} proved that the stack  $\sM^{v_{\orb}}(\sigma)$ is an Artin stack of finite type over $\cc$.
When $G$ is nontrivial,  the sheaves inside  the moduli stack $\sM^{v_{\orb}}(\sigma)$ of $G$-equivariant sheaves  must be $G$-invariant.  
Let $\sM^G$ be the moduli stack of $G$-fixed stable objects in the heart $\sA_{\omega}$. 
Let $\sM^{v_{\orb}}(\sigma)^G\subset \sM^G$ be the  corresponding 
 $G$-fixed objects in $\sM^{v_{\orb}}(\sigma)$ (i.e., forgetting about the $G$-equivariant structure).  Then 
$$\sM^{v_{\orb}}(\sigma)\to \sM^{v_{\orb}}(\sigma)^G$$
is a fibration with fiber the $G$-equivariant structures for a $G$-fixed sheaf.  The fixed part $\sM^{v_{\orb}}(\sigma)^G$
 is a closed  substack of the moduli stack of semistable objects in $S$ which is finite type over $\cc$.  Therefore the moduli stack  $\sM^{v_{\orb}}(\sigma)$ is also an Artin stack of finite type
over $\cc$.

\subsection{Joyce invariants $J^{\sigma}(v_{\orb})$}

Generalizing the construction in \S \ref{subsec_Hall_algebra}, and \S \ref{subsec_Joyce_invariant}, let $\hH(\sA_{\omega})$ be the Hall algebra of objects in the heart
$\sA_{\omega}$, which is an abelian category.

Still let $\sM(\sA_{\omega})$ be the moduli stack of objects in the abelian category $\sA_{\omega}$, which is an algebraic stack locally of finite type over $\cc$.
For the moduli stack $\sM^{v_{\orb}}(\sigma)$ corresponding to $\sigma=(\sZ, \sA_{\omega})$, we have an elment
$$\delta_{\sigma}(v_{\orb}):=[\sM^{v_{\orb}}(\sigma)\hookrightarrow \sM(\sA_{\omega})]\in \hH(\sA_{\omega})$$
Its logarithm is given by:
\begin{equation}\label{eqn_epsilon_sS_sigma}
\epsilon_{\sigma}(v_{\orb}):=
\sum_{\substack{\ell\ge 1, v_1+\cdots+v_{\ell}=v_{\orb}, v_i\in \Gamma_0^G \\
\arg\sZ(v_i)=\arg\sZ(v_{\orb})}}
\frac{(-1)^{\ell-1}}{\ell}\delta_{\sigma}(v_1)\star\cdots \star \delta_{\sigma}(v_{\ell})
\end{equation}
Then consider $P_q: \hH(\sA_{\omega})\to \qq(q^{\frac{1}{2}})$ and we define

\begin{defn}
$$J^{\sigma}(v_{\orb}):=\lim_{q^{\frac{1}{2}}\to 1}(q-1)\epsilon_{\sigma}(v_{\orb})\in \qq.$$
\end{defn}
Our main result is:

\begin{thm}\label{thm_Joyce_sigma}
The Joyce invariant $J^{\sigma}(v_{\orb})$ is independent to the stability conditions. Moreover
$$J^{\sigma}(v_{\orb})=J^{\omega}(v_{\orb}).$$
\end{thm}

\subsection{Gieseker stability and $\nu_{k, D}$-stability}\label{subsec_Bridgeland_some_sS}
We will recall the relations between the Gieseker stability and
$\nu_{k, D}$-stability.

We define the twisted Hilbert polynomial of a sheaf $E$ on $\sS$ as:
$$
G_{\omega,D}(E,m)=\frac{m^2}{2}\omega^2+m\omega\frac{\Ch_1^D(E)}{\rk(E)}+\frac{\Ch_2^D(E)}{\rk(E)}+\chi(\sO_{\sS})
$$

\begin{defn}
We say $E$ is $(\omega, D)$-twisted Gieseker (semi)stable (or
$G_{\omega,D}$ (semi)stable) if for all proper subsheaves
$F\hookrightarrow E$, we have $G_{\omega,D}(F,m)< (\leq)
G_{\omega,D}(E,m)$ for $m>>0$.
\end{defn}

We can write just  $G_{\omega,D}(E,m)$ and $\nu_{k,D}(E)$ in the following form:
\begin{align}\label{eqn_twisted_G}
G_{\omega,D}(E,m)&=\frac{m^2}{2}\omega^2+\frac{\omega\Ch_1^D(E)}{\rk(E)}m+\frac{\Ch_2^D(E)}{\rk(E)}+\chi(\sO_{\sS})\\
&=\frac{m^2}{2}\omega^2+\omega^2\mu_{\omega, D}(E)
m+\frac{\Ch_2^D(E)}{\rk(E)}+\chi(\sO_{\sS}) \nonumber
\end{align}
and
\begin{align}\label{eqn_Nu_k}
\nu_{k,D}(E,m)&=\frac{-\frac{k^2}{2}\omega^2 \Ch_0^D(E)}{\omega \Ch_1^D(E)}+\frac{\Ch_2^D(E)}{\omega \Ch_1^D(E)}\\
&=-\frac{1}{\mu_{\omega, D}(E)}\cdot
\frac{k^2}{2}+\frac{\Ch_2^D(E)}{\omega \Ch_1^D(E)}. \nonumber
\end{align}

We give the orbifold analogue of Proposition 6.4 and Proposition 6.5
in \cite{Toda_Adv}. We follow the strategy in \cite{Sun_H} given by
the third author.

\begin{prop}
For any object $E\in \Coh(\sS)\cap \sA_{\omega, D}$, there exists a
constant $N$ only depending on $E$ such that $E$ is
$G_{\omega,D}$-semistable if $E$ is $\nu_{k,D}$-semistable for some
$k\geq N$.
\end{prop}
\begin{proof}
We assume that $E$ is $\nu_{k,D}$-semistable for some $k>0$ but not
$G_{\omega,D}$-semistable. Then one can take a subsheaf $F\subset E$
such that $E/F$ is $G_{\omega,D}$-semistable and $G_{\omega,D}(F,
m)>G_{\omega,D}(E, m)$ for $m\gg0$. Hence  by the definition of
$G_{\omega,D}$ and $\sA_{\omega, D}$, we obtain two cases:

\begin{enumerate}
\item $\mu_{\omega, D}(F)>\mu_{\omega, D}(E)>0$ and $\mu_{\omega, D}(E/F)>0$;
\item $\mu_{\omega, D}(F)=\mu_{\omega, D}(E)>0$, $\mu_{\omega, D}(E/F)>0$ and $\frac{\Ch_2^D(E)}{\rk E}<\frac{\Ch_2^D(F)}{\rk
F}$.
\end{enumerate}
It is obvious that Case (2) contradicts the
$\nu_{k,D}$-semistability of $E$. Thus one gets $\mu_{\omega,
D}(E)<\mu_{\omega, D}(F)$. This implies that $\omega\Ch_1(E)\rk
F<\omega\Ch_1(F)\rk E$. Hence one sees that
\begin{eqnarray}\label{2.9.3}
\nonumber\omega\Ch^D_1(F)\rk E-\omega\Ch^D_1(E)\rk F&=&\omega\Ch_1(F)\rk E-\omega\Ch_1(E)\rk F\geq1;\\
\nonumber\frac{\omega^2\rk E}{\omega\Ch^D_1(E)}-\frac{\omega^2\rk
F}{\omega\Ch^D_1(F)}&\geq&\frac{\omega^2}{\omega\Ch^D_1(E)\cdot\omega\Ch^D_1(F)}\\
\nonumber&=&\frac{1}{\omega\Ch^D_1(E)\cdot\mu_{\omega,D}(F)\rk F}\\
&\geq&\frac{1}{\omega\Ch^D_1(E)\cdot\mu^+_{\omega,D}(E)\rk E}.
\end{eqnarray}
From the $\nu_{k,D}$-semistability of $E$, it follows that
$$\frac{\Ch_2^D(E)-\frac{k^2}{2}\omega^{2}\Ch_0^D(E)}{\omega\Ch_1^D(E)}\geq\frac{\Ch_2^D(F)-\frac{k^2}{2}\omega^{2}\Ch_0^D(F)}{\omega\Ch_1^D(F)}.$$
Combining this and (\ref{2.9.3}), one deduces
\begin{eqnarray}\label{2.9.4}
\nonumber-\frac{k^2/2}{\omega\Ch^D_1(E)\cdot\mu^+_{\omega,D}(E)\rk
E}&\geq&-\frac{k^2}{2}\left(\frac{\omega^2\rk
E}{\omega\Ch^D_1(E)}-\frac{\omega^2\rk F}{\omega\Ch^D_1(F)}\right)\\
\nonumber&\geq&\frac{\Ch_2^D(F)}{\omega\Ch_1^D(F)}-\frac{\Ch_2^D(E)}{\omega\Ch_1^D(E)}\\
\nonumber&\geq&\frac{\Ch_2^D(F)}{\omega^2\rk
F\cdot\mu^+_{\omega,D}(E)}-\frac{\Ch_2^D(E)}{\omega\Ch_1^D(E)}\\
&>&\frac{\Ch_2^D(F)}{\omega^2\rk
E\cdot\mu^+_{\omega,D}(E)}-\frac{\Ch_2^D(E)}{\omega\Ch_1^D(E)}.
\end{eqnarray}
On the other hand, since $E/F$ is $G_{\omega,D}$-semistable,
Bogomolov's inequality gives
\begin{eqnarray*}
\Ch_2^D(F)&=&\Ch_2^D(E)-\Ch_2^D(E/F)\\
&\geq&\Ch_2^D(E)-\frac{(\omega\Ch_1^D(E/F))^2}{2\omega^2\rk(E/F)}\\
&\geq&\Ch_2^D(E)-\frac{(\omega\Ch_1^D(E)-\omega\Ch_1^D(F))^2}{2\omega^2}\\
&>&\Ch_2^D(E)-\frac{(\omega\Ch_1^D(E))^2}{2\omega^2}
\end{eqnarray*}
From this and (\ref{2.9.4}), one infers that
$$-\frac{k^2/2}{\omega\Ch^D_1(E)\cdot\mu^+_{\omega,D}(E)\rk
E}>\frac{\Ch_2^D(E)}{\omega^2\rk
E\cdot\mu^+_{\omega,D}(E)}-\frac{(\omega\Ch_1^D(E))^2}{2(\omega^2)^2\rk
E\cdot\mu^+_{\omega,D}(E)}-\frac{\Ch_2^D(E)}{\omega\Ch_1^D(E)}.$$
This implies
$$k^2<\frac{(\omega\Ch^D_1(E))^3}{(\omega^2)^2}+2\rk E\left(\mu^+_{\omega,D}(E)-\mu_{\omega,D}(E)\right)\Ch_2^D(E).$$
Therefore, one completes the proof by taking
$$N=\sqrt{\frac{(\omega\Ch^D_1(E))^3}{(\omega^2)^2}+2\rk
E\left(\mu^+_{\omega,D}(E)-\mu_{\omega,D}(E)\right)\Ch_2^D(E)}.$$
\end{proof}

\begin{prop}
For any object $E\in \Coh(\sS)\cap \sA_{\omega, D}$, there exists a
constant $N$ only depending on $E$ such that $E$ is
$\nu_{k,D}$-semistable for any $k\geq N$ if $E$ is
$G_{\omega,D}$-semistable.
\end{prop}
\begin{proof}
The proof is a mimic of that of \cite[Theorem 1.3]{Sun_H}. We assume
that $E$ is not $\nu_{k,D}$-semistable for some $k>0$ but
$G_{\omega,D}$-semistable. Let $F$ be the $\nu_{k,D}$-maximal
subobject of $E$ in $\sA_{\omega, D}$. By \cite[Lemma 4.1]{Sun_H},
one sees that $\rk F\leq\rk E$ if
$$k\geq N_0:=\sqrt{\max\left\{\frac{\overline{\Delta}^D_{\omega}(E)}{\omega^2\rk E}-\frac{\mu^2_{\omega,D}(E)\rk E}{\omega^2+\rk E}, 0\right\}}.$$
Hence $F$ is a subsheaf of $E$ when $k\geq N_0$. The
$G_{\omega,D}$-semistability of $E$ gives two cases:
\begin{enumerate}
\item $\mu_{\omega, D}(E)>\mu_{\omega, D}(F)>0$;
\item $\mu_{\omega, D}(E)=\mu_{\omega, D}(F)>0$ and $\frac{\Ch_2^D(E)}{\rk E}\geq\frac{\Ch_2^D(F)}{\rk
F}$.
\end{enumerate}
It is obvious that Case (2) contradicts that $E$ is not
$\nu_{k,D}$-semistable. In Case (1), one obtains $\omega\Ch_1(E)\rk
F>\omega\Ch_1(F)\rk E$. Hence one sees that
\begin{eqnarray*}
\omega\Ch^D_1(E)\rk F-\omega\Ch^D_1(F)\rk E&=&\omega\Ch_1(E)\rk
F-\omega\Ch_1(F)\rk E\geq1.
\end{eqnarray*}
It implies that
\begin{eqnarray}\label{2.9.5}
\nonumber\frac{\omega^2\rk F}{\omega\Ch^D_1(F)}-\frac{\omega^2\rk
E}{\omega\Ch^D_1(E)}&\geq&\frac{\omega^2}{\omega\Ch^D_1(E)\cdot\omega\Ch^D_1(F)}\\
\nonumber&=&\frac{1}{\omega\Ch^D_1(E)\cdot\mu_{\omega,D}(F)\rk F}\\
&>&\frac{1}{\omega\Ch^D_1(E)\cdot\mu_{\omega,D}(E)\rk E}.
\end{eqnarray}
Since $E$ is not $\nu_{k,D}$-semistable, we have
$$\frac{\Ch_2^D(E)-\frac{k^2}{2}\omega^{2}\Ch_0^D(E)}{\omega\Ch_1^D(E)}<\frac{\Ch_2^D(F)-\frac{k^2}{2}\omega^{2}\Ch_0^D(F)}{\omega\Ch_1^D(F)}.$$
Combining this and (\ref{2.9.5}), one deduces
\begin{eqnarray}\label{2.9.6}
\nonumber\frac{k^2/2}{\omega\Ch^D_1(E)\cdot\mu_{\omega,D}(E)\rk
E}&<&\frac{k^2}{2}\left(\frac{\omega^2\rk
F}{\omega\Ch^D_1(F)}-\frac{\omega^2\rk
E}{\omega\Ch^D_1(E)}\right)\\
&<&\frac{\Ch_2^D(F)}{\omega\Ch_1^D(F)}-\frac{\Ch_2^D(E)}{\omega\Ch_1^D(E)}.
\end{eqnarray}

On the other hand, since $F$ is $\nu_{k,D}$-semistable, Bogomolov's
inequality gives
\begin{eqnarray*}
\frac{\Ch_2^D(F)}{\omega\Ch_1^D(F)}&\leq&\frac{\omega\Ch_1^D(F)}{2\omega^2\rk
F}\\
&<&\frac{1}{2}\mu_{\omega,D}(E)
\end{eqnarray*}
From this and (\ref{2.9.6}), one obtains
\begin{eqnarray*}
k^2&<&\omega\Ch^D_1(E)\mu_{\omega,D}(E)\rk
E\left(\mu_{\omega,D}(E)-\frac{2\Ch_2^D(E)}{\omega\Ch^D_1(E)}\right)\\
&=&\mu^3_{\omega,D}(E)\omega^2(\rk E)^2-2\mu_{\omega,D}(E)(\rk
E)\Ch_2^D(E).
\end{eqnarray*}
We finish the proof by taking
$$N=\max\left\{\sqrt{\mu^3_{\omega,D}(E)\omega^2(\rk E)^2-2\mu_{\omega,D}(E)(\rk
E)\Ch_2^D(E)}, N_0\right\}.$$

\end{proof}

The above two propositions gives an equivalence between
$G_{\omega,D}$-stability and $\nu_{k,D}$-stability:
\begin{thm}\label{thm_G_Nu}
For any $E\in \Coh(\sS)\cap \sA_{\omega, D}$, there exists a
constant $N$ only depending on $E$ such that $E$ is
$\nu_{k,D}$-semistable for $k\geq N$ if and only if $E$ is
$G_{\omega,D}$-semistable.
\end{thm}

\subsection{Proof of Theorem \ref{thm_Joyce_sigma}}\label{subsec_proof_Joyce_sigma}

From \cite[Theorem 6.6]{Toda_Adv}, it is enough to compare $J^{\sigma_{kD}}(v_{\orb})$ for $\sigma_{kD}=(\sZ_{k,D},\sA_D)$ and $J^{\omega}(v_{\orb})$
for $k>>0$.
From the construction of  (\ref{eqn_epsilon_sS}) and   (\ref{eqn_epsilon_sS_sigma}) before,  after taking Joyce invariants we have
$$
J^{\omega}(v_{\orb})=\sum_{\ell\ge 1, v_1+\cdots+v_{\ell}=v_{\orb}}\frac{(-1)^{\ell-1}}{\ell}\prod_{i=1}^{\ell}J^{\omega}(v_i)
$$
and
$$
J^{\sigma_{kD}}(v_{\orb})=\sum_{\ell\ge 1, v_1+\cdots+v_{\ell}=v_{\orb}}\frac{(-1)^{\ell-1}}{\ell}\prod_{i=1}^{\ell}J^{\sigma_{kD}}(v_i)
$$
From Theorem \ref{thm_G_Nu} in \S \ref{subsec_Bridgeland_some_sS}, we have
$$\sM^{v_i}(\sigma_{kD})\cong \sM^{v_i}(\omega).$$
Thus we have $\prod_{i=1}^{\ell}J^{\omega}(v_i)=\prod_{i=1}^{\ell}J^{\sigma_{kD}}(v_i)$.

\section{Sheaves on local orbifold K3 surfaces}\label{sec_local_orbifold_K3}

\subsection{Crepant resolutions}

Recall the surface Deligne-Mumford stack $\sS=[S/G]$ in \S \ref{subsec_quotient_K3}, where $G$, as a finite group, acts as symplectic morphisms on $S$.  The stacky points of $\sS$ consists of
ADE type orbifold points.  Let
$$f: Y\rightarrow \overline{\sS}=S/G$$
be the minimal resolution of $S/G$.  It is a crepant resolution, and  $Y$ is also a smooth K3 surface.
The exceptional curves of $Y$ over each orbifold singular point $P\in S/G$ are given by ADE type Dynkin diagrams.  More details can be found in \cite{Huybrechts}.
Consider the diagram:
\[
\xymatrix{
&Z\ar[dl]\ar[dr]\\
[S/G]\ar[dr]_{p}&& Y\ar[dl]^{f}\\
&S/G&
}
\]
such that $[S/G]\dashrightarrow Y$ is a crepant birational morphism. This is the situation in \cite{BKR}, where $Y$ is one irreducible component in the $G$-Hilbert scheme $G-\Hilb$, and
$Z\subset Y\times S$ is the universal subscheme.  Therefore from \cite[Theorem1.2]{BKR}, there
is an equivalence
$$\Phi: D(\Coh(\sS))\stackrel{\sim}{\longrightarrow} D(\Coh(Y))$$
between derived categories.

\begin{prop}\label{prop_FM_sS_Y}
We have the following commutative diagram:
\begin{equation}\label{eqn_diagram_sS_Y}
\xymatrix{
D(\Coh(\sS))\ar[r]^{\Phi}\ar[d]& D(\Coh(Y))\ar[d]\\
K(\Coh(\sS))\ar[r]^{\Phi}\ar[d]_{\widetilde{\Ch}\cdot \sqrt{\widetilde{\td}_{\sS}}}& K(\Coh(Y))\ar[d]^{\Ch\cdot \sqrt{\td_Y}}\\
H^*_{\CR}(\sS)\ar[r]^{\Phi_*} & H^*(Y)
}
\end{equation}
such that it induces an isomorphism
$$\Phi_*: H^*_{\CR}(\sS)\to H^*(Y)$$
between the Chen-Ruan cohomology of $\sS$ and the cohomology space of $Y$.
\end{prop}
\begin{proof}
This is the result in \cite{BKR}, and known result for the cohomology of the crepant resolution and the  Chen-Ruan cohomology of the stack
$\sS=[S/G]$.
\end{proof}

\subsection{Local orbifold K3 surfaces}\label{subsec_local_K3}

For the Calabi-Yau surface Deligne-Mumford stack  $\sS=[S/G]$, and the K3 surface $Y$, we take
$$\XX:=\sS\times \cc; \quad  Z=Y\times \cc.$$
Here $Z$ is called the local K3 surface and $\XX$ is called the local orbifold K3 surface.  $\XX$ is a smooth Calabi-Yau threefold Deligne-Mumford stack.
Their natural compactifications are given by
$$\overline{\XX}=\sS\times \pp^1; \quad  \overline{Z}=Y\times \pp^1. $$
Let
$$\pi:  \overline{\XX}=\sS\times \pp^1\to \pp^1; \quad \pi: \overline{Z}=Y\times \pp^1\to \pp^1$$
be projections.  We consider the abelian subcategories
$$\Coh_{\pi}(\overline{\XX})\subset \Coh(\overline{\XX}); \quad \Coh_{\pi}(\overline{Z})\subset \Coh(\overline{Z})$$
to be the subcategories consisting of sheaves supported on the fibers of $\pi$.  We denote by:
\begin{equation}\label{eqn_sD_sS_Y}
\sD_0^{\sS}:=D^b(\Coh_{\pi}(\overline{\XX})); \quad \sD_0^{Y}:=D^b(\Coh_{\pi}(\overline{Z}))
\end{equation}
the corresponding derived categories of $\Coh_{\pi}(\overline{\XX}), \Coh_{\pi}(\overline{Z})$ respectively.
We define:

\begin{defn}
$$\sD^{\sS}:=\langle \pi^*\Pic(\pp^1), \Coh_{\pi}(\overline{\XX})\rangle_{\tr}\subset D^b(\Coh(\overline{\XX}))$$
$$\sD^Y:=\langle \pi^*\Pic(\pp^1), \Coh_{\pi}(\overline{Z})\rangle_{\tr}\subset D^b(\Coh(\overline{Z})).$$
\end{defn}

\subsection{Chern characters}\label{subsec_Chern_character}

We introduce the Chern characters on the categories in \S \ref{subsec_local_K3}.
Let
$$\pi_1: \overline{\XX}=\sS\times \pp^1\to \sS$$
be the first projection morphism.

\begin{defn}
Define the homomorphisms
\begin{equation}\label{eqn_cl0_3}
\widetilde{\cl}_0: K(\sD_0^{\sS})\stackrel{\pi_{1*}}{\rightarrow} K(\sS)\stackrel{\widetilde{\Ch}}{\rightarrow} \Gamma_0^G
\end{equation}
and
\begin{equation}\label{eqn_vorb_3}
v_{\orb}: K(\sD_0^{\sS})\stackrel{\pi_{1*}}{\rightarrow} K(\sS)\stackrel{\widetilde{\Ch}\cdot \sqrt{\widetilde{\td}_{\sS}}}{\longrightarrow} \Gamma_0^G
\end{equation}
\end{defn}

For the triangulated category $\sD^{\sS}$, we have
$$\Gamma^G:=H^0(\overline{\XX})\oplus (\Gamma_0^G\boxtimes H^2(\pp^1,\qq))\subset H^*_{\CR}(\overline{\XX},\qq)$$
Thus we have a group homomorphism
$$\widetilde{\cl}:=\widetilde{\Ch}:  K(\sD^{\sS})\to \Gamma^G$$
We can write $\Gamma^G$ as:
$$\Gamma^G=\qq\oplus\qq\oplus NS(\sS)\oplus \qq^{|I_1\sS|}\oplus \qq$$
and $v_{\orb}\in \Gamma^G$ is given by $v_{\orb}=(R, r, \widetilde{\beta}, n)$ such that $(r, \widetilde{\beta}, n)\in \Gamma_0^G$ and $\widetilde{\beta}\in NS(\sS)\oplus \qq^{|I_1\sS|}$.

For the K3 surface $Y$, and $\overline{Z}=Y\times \pp^1$, we have similar Chern character morphisms as in \cite[\S 2.3]{Toda_JDG}:
$$\widetilde{\cl}_0: K(\sD_0^{Y})\stackrel{\pi_{1*}}{\rightarrow} K(Y)\stackrel{\Ch}{\rightarrow} \Gamma_0^Y,$$
where
$\Gamma_0^Y\cong \zz\oplus NS(Y)\oplus \zz$, and
\begin{equation}\label{eqn_Mukai_Y}
v: K(\sD_0^{Y})\stackrel{\pi_{1*}}{\rightarrow} K(Y)\stackrel{\Ch\cdot \sqrt{\td_Y}}{\rightarrow} \Gamma_0^Y
\end{equation}
Also for
$\sD^Y:=\langle \pi^*\Pic(\pp^1), \Coh_{\pi}(\overline{Z})\rangle_{\tr}$, we have
$$\cl=\Ch: K(\sD_0^{Y})\to \Gamma_0^Y$$
where $\Gamma^Y=\zz\oplus \Gamma_0^Y$.

\subsection{Joyce invariants in $\Coh_{\pi}(\overline{\XX})$}

Still let $\pi: \XX=\sS\times \cc\to \cc$ be the projection. Let $\Coh_{\pi}(\XX)\subset \Coh_{\pi}(\overline{\XX})$ be the subcategory of sheaves supported on the fibers on
$\pi: \XX\to \cc$.

Let $\sM_{\pi}(\XX)$ be the stack of objects in $\Coh_{\pi}(\XX)$, and this stack is an algebraic stack locally of finite type over $\cc$. Similarly as in \S \ref{subsec_Joyce_invariant}, let
$\hH(\Coh_{\pi}(\XX))$ be the Hall algebra of the category $\Coh_{\pi}(\XX)$. Let $\sM_{\omega, \XX}(v_{\orb})$ be the moduli stack of $\omega$-Gieseker semistable sheaves
$E\in \Coh_{\pi}(\XX)$ satisfying $v_{\orb}(E)=v_{\orb}$ as in (\ref{eqn_vorb_3}).  The stack  $\sM_{\omega, \XX}(v_{\orb})$ is an algebraic stack of finite type over $\cc$.
Thus there is an element
$$\delta_{\omega, \XX}(v_{\orb}):=
[\sM_{\omega, \XX}(v_{\orb})\hookrightarrow \sM_{\pi}(\XX)]\in \hH(\Coh_{\pi}(\XX))$$
Its logarithm is given by:
$$\epsilon_{\omega, \XX}(v_{\orb}):=
\sum_{\substack{\ell\ge 1, v_1+\cdots+v_{\ell}=v_{\orb}, v_i\in \Gamma_0^G\\
\overline{\chi}_{\omega,v_i}(m)=\overline{\chi}_{\omega,v_{\orb}}(m)}}\frac{(-1)^{\ell-1}}{\ell}\delta_{\omega,\XX}(v_1)\star\cdots\star\delta_{\omega,\XX}(v_{\ell})
$$

Then we can define the Joyce invariants:

\begin{defn}\label{defn_Joyce_invariants_XX}
Let
$$C(\XX):=\Im(v_{\orb}: \Coh_{\pi}(\XX)\to \Gamma_0^G)$$
We define the Joyce Invariants:
If $v_{\orb}\in C(\XX)$,
$$J^{\omega}(v_{\orb}):=\lim_{q^{\frac{1}{2}\to 1}}(q-1)\cdot P_q(\epsilon_{\omega,\XX}(v_{\orb}))$$

If $-v_{\orb}\in C(\XX)$,
$J^{\omega}(-v_{\orb})$.  $J^{\omega}(v_{\orb})=0$ otherwise.
\end{defn}

Similar to the case of K3 surfaces, the Joyce invariants $J^{\omega}(v_{\orb})$ is independent to the polarization $\omega$.

\subsection{A digression on Hilbert scheme of points on $\sS$}\label{subsec_Hilbert_scheme}

We talk about the Hilbert scheme $\Hilb^n(\sS)$ of zero dimensional substacks in  the Deligne-Mumford surface $\sS$.
A good reference can be found in \cite{BG}.  First we have
$$\Hilb(\sS)\cong \Hilb(S)^{G},$$
i.e., the Hilbert scheme of zero dimensional substacks of $\sS$ is naturally identified with the $G$-fixed Hilbert scheme of $S$.
More detail explanation of the zero dimensional substacks supported on the stacky points of $\sS$ can be found in \cite[\S 2]{BG}.

The components of $\Hilb^n(\sS)$ are given by the $K$-theory class of $\sO_{\sT}$ for $\sT\subset \sS$ the zero dimensional substack of $\sS$.
Let $P_1, \cdots, P_r\in S/G$ be the singular points where the stabilizer subgroups $G_i\subset G$ have orders $k_i$ and ADE type $\Delta(i)$ ($\Delta(i)$ is the corresponding root system (ADE type)).
One can write down the $K$-theory class as:
\begin{equation}\label{eqn_K-theory_class_sT}
[\sO_{\sT}]=n[\sO_P]+\sum_{i=1}^{r}\sum_{j=1}^{n(i)}m_j(i)[\sO_{P_i}\otimes \rho_j(i)]
\end{equation}
where $P\in \sS$ is a generic point. Around the singular point $P_i\in S/G$, we have $G_{i}\subset \SU(2)$ and $\Delta_i$ has rank $n(i)$.
Here $\rho_0(i), \rho_1(i), \cdots, \rho_{n(i)}(i)$ are the irreducible representations of $G_i$.
Let $\mathfrak{m}=\{m_j(i)\}$ and let $\Hilb^{n, \mathfrak{m}}(\sS)\subset \Hilb(\sS)$ be the component with respect to the $K$-theory class (\ref{eqn_K-theory_class_sT}).
We let
$$D_{\mathfrak{m}}:=\sum_{i=1}^{r}\sum_{j=1}^{n(i)}m_j(i)\cdot E_j(i)$$
where $E_1(i), \cdots, E_{n(i)}(i)$ are the exceptional curves over $P_i$ under the crepant resolution $Y\to S/G$.  As in \cite[\S 5]{BG}, we write
$\mathfrak{m}=\{m_j(i)\}$ as the vector $\mathbf{m}(i)\in M_{\Delta(i)}$ in the root lattice. We have
$$D^2_{\mathfrak{m}}=-\sum_{i=1}^{r}\left(\mathbf{m}(i)|\mathbf{m}(i)\right)_{\Delta(i)}$$
From \cite[Proposition 5.1]{BG}, we have
\begin{lem}\label{lem_Hilbert_sS_Y}
There exists a birational morphism between the  Hilbert scheme $\Hilb^{n, \mathfrak{m}}(\sS)$
and the Hilbert scheme $\Hilb^{n+\frac{1}{2}D^2_{\mathfrak{m}}}(Y)$.
\end{lem}

Recall the isomorphism $\Phi_*:  H^*_{\CR}(\sS)\to H^*(Y)$ in Proposition \ref{prop_FM_sS_Y}, if there is a Mukai vector
$v_{\orb}\in \Gamma_0^G$, then
$v_Y:=\Phi_*(v_{\orb})$ is a Mukai vector in $\Gamma_0^Y$.
We can write $v_{\orb}\in \Gamma_0^G$ as
$$v_{\orb}=(r, (\beta, \mathfrak{m}), n)$$
where $\mathfrak{m}=\{m_j(i)\}$ corresponding to the stacky points $P_1, \cdots, P_r$. Under the crepant resolution morphism
$$\sigma: Y\to S/G$$
we have
$m_j(i)[\sO_{P_i}\otimes \rho_j(i)]$ correspond to $m_j(i)[E_j(i)]$.
Let
$$n:=\langle v_Y, v_Y\rangle/2+1-\frac{1}{2}D^2_{\mathfrak{m}}.$$
Since under the isomorphism $\Phi_*: H^*_{\CR}(\sS)\to H^*(Y)$,
$\langle v_{\orb}, v_{\orb}\rangle=\langle v_Y, v_Y\rangle$.

\begin{lem}\label{lem_vorb_vY}
For any Mukai vector $v_{\orb}=(r, (\beta, \mathfrak{m}), n)\in \Gamma_0^G$ such that
$v_Y=\Phi_*(v_{\orb})$.  There is a birational morphism
$$\Hilb^{n, \mathfrak{m}}(\sS)\dasharrow  \Hilb^{\langle v_Y, v_Y\rangle/2+1}(Y).$$
\end{lem}
\begin{proof}
This is from Lemma \ref{lem_Hilbert_sS_Y}.
\end{proof}

\subsection{Multiple cover formula}

Recall the Joyce invariant $J^{\omega}(v_{\orb})$ for $v_{\orb}\in \Gamma_0^G$ in Definition \ref{defn_Joyce_invariants_XX}.  Since the invariant is independent to the
polarization $\omega$, we just write the Joyce invariant as $J(v_{\orb})$.

\begin{thm}\label{thm_multiple_cover}
There is a multiple cover formula for $J(v_{\orb})$:
$$J(v_{\orb})=\sum_{k|v_{\orb}, k\ge 1}
\frac{1}{k^2}\chi(\Hilb^{n, \mathfrak{m}}(\sS))
$$
where  the data $n, \mathfrak{m}$ are determined by
$\frac{1}{k}v_{\orb}=(r, (\beta, \mathfrak{m}), n)$.
\end{thm}

We prove Theorem \ref{thm_multiple_cover} in the following sections.

\subsection{Bridgeland stability conditions on $\sD_0^{\sS}$ and $\sD_0^Y$}

On the category $\Coh_{\pi}(\overline{\XX})$ (or $\Coh_{\pi}(\overline{Z})$), we have the classical slope stability as in \S \ref{subsec_Bridgeland_stability}.
For $E\in \Coh_{\pi}(\overline{\XX})$,
$$\mu_{\omega}(E)=\frac{\omega\cdot c_1(E)}{\rk(E)}.$$
We still have the maximal slope $\mu_{\omega}^+(E)$ and minimal slope $\mu_{\omega}^-(E)$ in the Harder-Narasimhan filtration of $E$.
We also have the torsion pair $(\sF_{\omega}, \sT_{\omega})$ as in (\ref{eqn_torsion_pair}). Let
$$\sB_{\omega}:=\langle \sF_{\omega}[1], \sT_{\omega}\rangle\subset \sD_0^{\sS}$$
Then $\sB_{\omega}$ is the heart of a bounded $t$-structure on $\sD_0^{\sS}$.  Note that replacing $\omega$ by $t\omega$ does not change $\sB_{\omega}$ for $t>0$.
Let us define
$$\sZ_{t\omega}: K(\sB_{\omega})\to \cc$$
by
$$E\mapsto \int_{\sS}e^{-it\omega}\Ch(E)$$
where $\Ch(E)=(\Ch_0(E), \Ch_1(E), \Ch_2(E))\in H^*(\sS)$ by the general Chern character.  Then
$$\sZ_{t\omega}(E)=-\Ch_2(E)+\frac{t^2 \omega^2}{2}\Ch_0(E)+i t\omega \Ch_1(E).$$
The pair $\sigma_{t\omega}=(\sZ_{t\omega}, \sB_{\omega})$ is a Bridgeland stability condition, i.e.,
$$\sigma_{t\omega}=(\sZ_{t\omega}, \sB_{\omega})\in \Stab_{\Gamma_0^G}(\sD_0^{\sS})$$
where $\Stab_{\Gamma_0^G}(\sD_0^{\sS})$ is the Bridgeland stability manifold.

Let $\hH(\sB_{\omega})$ be the Hall algebra of the abelian category $\sB_{\omega}$.  Let $\sM_{t\omega}(v_{\orb})\subset \sM(\sB_{\omega})$ be the moduli substack
of $\sZ_{t\omega}$-semistable objects $E\in \sB_{\omega}$ with $\widetilde{\cl}(E)=v_{\orb}$.  Then we have an element
$$\delta_{\sigma_{t\omega}}(v_{\orb}):=
[\sM_{t\omega}(v_{\orb})\hookrightarrow \sM(\sB_{\omega})]\in \hH(\sB_{\omega})$$
and its logarithm
\begin{equation}\label{eqn_epsilon_XX_sigma}
\epsilon_{\sigma_{t\omega}}(v_{\orb}):=
\sum_{\substack{\ell\ge 1, v_1+\cdots+v_{\ell}=v_{\orb}, v_i\in \Gamma_0^G \\
\arg\sZ_{t\omega}(v_i)=\arg\sZ_{t\omega}(v_{\orb})}}
\frac{(-1)^{\ell-1}}{\ell}\delta_{\sigma_{t\omega}}(v_1)\star\cdots \star \delta_{\sigma_{t\omega}}(v_{\ell})
\end{equation}

We let
$$C(\sB_{\omega}):=\Im(\widetilde{\cl}_0: \sB_{\omega}\to \Gamma_0^G)$$

\begin{defn}\label{defn_Joyce_invariants_XX}
For $v_{\orb}\in C(\sB_{\omega})$, we define
$$\overline{J}^{\sigma_{t\omega}}(v_{\orb})=\lim_{q^{\frac{1}{2}\to 1}}(q-1)P_q(\epsilon_{\sigma_{t\omega}}(v_{\orb})).$$

For  $-v_{\orb}\in C(\sB_{\omega})$, we define $\overline{J}^{\sigma_{t\omega}}(v_{\orb})=\overline{J}^{\sigma_{t\omega}}(-v_{\orb})$.

For $v_{\orb}\notin C(\sB_{\omega})$, $\overline{J}^{\sigma_{t\omega}}(v_{\orb})=0$.
\end{defn}

\subsection{The invariant $\overline{J}^{\omega}(v_{\orb})$}

We work on $\Coh_{\pi}(\overline{\XX})$. Similarly from Definition \ref{defn_Joyce_invariants_XX}, we can define Joyce invariant
$$\overline{J}^{\omega}(v_{\orb})\in\qq$$
counting $\omega$-Gieseker semistable sheaves $E\in \Coh_{\pi}(\overline{\XX})$ with $v_{\orb}(E)=v_{\orb}\in \Gamma_0^G$.  Also from
\cite[Theorem 4.21]{Toda_JDG}, $\overline{J}^{\omega}(v_{\orb})$ does not depend on $\omega$.  We have

\begin{lem}
For any $v_{\orb}\in \Gamma_0^G$,
$$\overline{J}^{\omega}(v_{\orb})=\overline{J}^{\sigma_{t\omega}}(v_{\orb}).$$

For any $v_{Y}\in \Gamma_0^Y$,
$$\overline{J}^{\omega}(v_{Y})=\overline{J}^{\sigma_{t\omega}}(v_{Y}).$$
\end{lem}
\begin{proof}
This is from Theorem \ref{thm_Joyce_sigma} proved in \S \ref{subsec_proof_Joyce_sigma}, and
\cite[Theorem 4.24]{Toda_JDG}.
\end{proof}

Recall that $\overline{J}^{\omega}(v_{\orb})$ is defined as
$$\lim_{q^{\frac{1}{2}}\to 1}(q-1)P_q(\epsilon_{\omega, \overline{\XX}}(v_{\orb}))$$
where
$$\epsilon_{\omega, \overline{\XX}}(v_{\orb}):=
\sum_{\substack{\ell\ge 1, v_1+\cdots+v_{\ell}=v_{\orb}, v_i\in \Gamma_0^G\\
\overline{\chi}_{\omega,v_i}(m)=\overline{\chi}_{\omega,v_{\orb}}(m)}}\frac{(-1)^{\ell-1}}{\ell}\delta_{\omega,\overline{\XX}}(v_1)\star\cdots\star\delta_{\omega,\overline{\XX}}(v_{\ell})
$$

Also since the invariants  $\overline{J}^{\omega}(v_{\orb})$ also are independent to the stability conditions.
We just write them as $\overline{J}(v_{\orb})$.
Then the same arguments as in \cite[Lemma 4.25, Lemma 4.26]{Toda_JDG} show that
\begin{lem}\label{lem_overlineJ_J}
We have
$$\overline{J}(v_{\orb})=2 J(v_{\orb})$$
\end{lem}

Similar relations hold when we replace $\overline{J}(v_{\orb})$ and $J(v_{\orb})$ by
$\overline{J}(v_{Y})$ and $J(v_{Y})$ respectively.

\subsection{Automorphic property}\label{subsec_automorphic}

We prove some automorphic property of $J(v_{\orb})$.  Let us write down the derived equivalence:
$$\Phi: D(\Coh(\sS))\stackrel{\sim}{\longrightarrow} D(\Coh(Y))$$
explicitly from \cite{BKR}.  The equivalence $\Phi$ is given by:
$$\Phi(-)=Rp_{2*}p_1^*(-\otimes^{L}\sE)$$
for the kernel $\sE\in D(\sS\times Y)$ of $\Phi$, where
$p_1: \sS\times Y\to \sS$
and
$p_2: \sS\times Y\to Y$ are the projections.
Thus we have the diagram (\ref{eqn_diagram_sS_Y}) before,
such that
$$\Phi_{*}(-)=Ip_{2*}(Ip_{1}^*(-)\cdot \widetilde{\Ch}(\sE)\cdot \sqrt{\widetilde{\td}_{\sS\times Y}})$$
where
$$Ip_1: I\sS\times Y\to I\sS$$ and
$$Ip_2: I\sS\times Y\to Y$$
are projections on inertia stacks.
$\Phi_*$ induces an isomorphism on the weight two Hodge structures.  Note that in general $\widetilde{H}_{\CR}(\sS)$, taken as the Chen-Ruan cohomology, will have
$\qq$- or $\cc$-coefficients.   Since we take
$$v_{\orb}: K(\Coh(\sS))\to \widetilde{H}_{\CR}(\sS)$$
by $$E\mapsto \widetilde{\Ch}(E)\sqrt{\widetilde{\td}_{\sS}}$$
We call $\widetilde{H}_{\CR}(\sS)$ an integral structure since from Fourier-Mukai pairing,
$$\chi(E,F)=-\langle v_{\orb}(E), v_{\orb}(F)\rangle.$$
And $\Phi_{*}$ should be an isomorphism from $\widetilde{H}_{\CR}(\sS)$ to $\widetilde{H}(Y,\zz)$.

The derived equivalence $\Phi$ induces an isomorphism on the stability manifolds:
$$\Phi_{\sst}:  \Stab(D(\Coh(\sS)))\stackrel{\sim}{\longrightarrow}\Stab(D(\Coh(Y))).$$

\begin{prop}\label{prop_Jbar_sS_Y}
Let $\Stab^{\circ}(D(\Coh(\sS)))$ and $\Stab^{\circ}(D(\Coh(Y)))$ be the connected components containing the stability conditions $\sigma_{t\omega}$
constructed before.  Then $\Phi_{\sst}$ takes $\Stab^{\circ}(D(\Coh(\sS)))$  to $\Stab^{\circ}(D(\Coh(Y)))$.
For any $v_{\orb}\in \Gamma_0^G$, we have
$$\overline{J}_{\sS}(v_{\orb})=\overline{J}_{Y}(\Phi_{*}v_{\orb})=\overline{J}_Y(v_{Y}).$$
\end{prop}
\begin{proof}
The proof is similar to \cite[Proposition 4.29]{Toda_JDG}.
Consider
$\overline{\XX}=\sS\times \pp^1$ and $\overline{Z}=Y\times \pp^1$. Then the equivalence $\Phi: D(\Coh(\sS))\stackrel{\sim}{\longrightarrow} D(\Coh(Y))$ induces an equivalence
$$\widetilde{\Phi}: D^b(\Coh(\overline{\XX}))\stackrel{\sim}{\longrightarrow} D^b(\Coh(\overline{Z}))$$
such that the kernel is given by:
$$\sE\boxtimes \sO_{\Delta_{\pp^1}}\in D^b(\Coh(\sS\times Y\times \pp^1\times\pp^1)).$$
Thus $\widetilde{\Phi}$ restricts to give an equivalence:
between $\sD_0^{\sS}$ and $\sD_0^Y$.  Therefore the diagram  (\ref{eqn_diagram_sS_Y})  gives a diagram:
\begin{equation}\label{eqn_diagram_sDD}
\xymatrix{
\sD_0^{\sS}\ar[r]^{\widetilde{\Phi}} \ar[d]_{\widetilde{\cl}+0\sqrt{\widetilde{\td}_{\sS}}}& \sD_0^Y\ar[d]^{\cl_0\sqrt{\td_Y}}\\
\Gamma_0^G\ar[r]^{\Phi_*}& \Gamma_0^Y
}
\end{equation}
Also $\widetilde{\Phi}$ induces the isomorphism
$$\widetilde{\Phi}_{\sst}: \Stab^{\circ}_{\Gamma_0^G}(\sD_0^{\sS})\stackrel{\sim}{\longrightarrow} \Stab^{\circ}_{\Gamma_0^Y}(\sD_0^{Y}).$$
Take $\sigma_{\sS}\in \Stab^{\circ}_{\Gamma_0^G}(\sD_0^{\sS})$, such that $\widetilde{\Phi}_{\sst}(\sigma_{\sS})=\sigma_Y$.
For any $v_{\orb}\in \Gamma_0^G$, from diagram (\ref{eqn_diagram_sDD}), we calculate
\begin{align*}
\overline{J}_{Y}(\Phi_*v_{\orb})&=\overline{J}_{Y}^{\sigma_Y}(\Phi_*v_{\orb})\\
&=\overline{J}_{Y}^{\widetilde{\Phi}_{\sst}(\sigma_{\sS})}(\Phi_*v_{\orb})\\
&=\overline{J}_{\sS}^{\sigma_{\sS}}(v_{\orb})\\
&=\overline{J}_{\sS}(v_{\orb}).
\end{align*}
\end{proof}

From Lemma \ref{lem_overlineJ_J},

\begin{cor}\label{cor_J_sS_Y}
We have:
$$J_{\sS}(v_{\orb})=J_Y(v_Y). $$
\\
\end{cor}

\subsection{Proof of the multiple cover formula Theorem \ref{thm_multiple_cover}}\label{subsec_proof_multiple}

In this section we prove the multiple cover formula for the invariants $J(v_{\orb})$ for $v_{\orb}\in \Gamma_0^G$.

First for the Mukai vector $v_Y\in \Gamma_0^Y$, since $Y$ is a smooth K3 surface, Toda's multiple cover formula, which was proved in \cite{MT}, says
$$J(v_Y)=\sum_{k| v_Y, k\ge 1}\frac{1}{k^2}\chi(\Hilb^{\langle v_Y, v_Y\rangle/2+1}(Y)).$$
Since $J(v_Y)=J(v_{\orb})$ for $\Phi_*(v_{\orb})=v_Y$, we have
$$J(v_{\orb})=\sum_{k| v_Y, k\ge 1}\frac{1}{k^2}\chi(\Hilb^{\langle v_Y, v_Y\rangle/2+1}(Y)).$$
From \S \ref{subsec_Hilbert_scheme} and Lemma \ref{lem_vorb_vY},  we have
$$\Hilb^{n, \mathfrak{m}}(\sS)\dasharrow \Hilb^{\langle \frac{v_Y}{k}, \frac{v_Y}{k}\rangle/2+1}(Y)$$
is birational equivalent, where
$$n:=\langle  \frac{v_Y}{k}, \frac{v_Y}{k}\rangle/2+1-\frac{1}{2}D^2_{\mathfrak{m}}=\langle  \frac{v_{\orb}}{k}, \frac{v_{\orb}}{k}\rangle/2+1-\frac{1}{2}D^2_{\mathfrak{m}}$$
and
$\mathfrak{m}$ is determined by the Mukai vector $\frac{v_{\orb}}{k}=(r, (\beta, \mathfrak{m}), n)\in \Gamma_0^G$.
Therefore we have
$$J(v_{\orb})=\sum_{k|v_{\orb}, k\ge 1}
\frac{1}{k^2}\chi(\Hilb^{n, \mathfrak{m}}(\sS)).
$$




\end{document}